\documentclass[11pt]{amsart}

\usepackage{lmodern}
\usepackage{microtype}
\usepackage[english]{babel}
\usepackage{mathtools}
\usepackage{hyperref}
\usepackage[hyperpageref]{backref}
\usepackage[msc-links]{amsrefs}

\usepackage{pgfplots}
\pgfplotsset{compat=1.18}

\theoremstyle{plain} 
\newtheorem*{theorem}{Theorem}
\newtheorem*{lemma}{Lemma}
\newtheorem*{corollary}{Corollary}

\begin{document} 
\title{A bootstrap proof of Bohr's theorem}

\begin{abstract}
	Let $f$ be an analytic function in the unit disc whose modulus is bounded by $1$. We present a new proof of Bohr's theorem, which asserts that its majorant function $Mf(r)$ is bounded by $1$ for $0\leq r \leq 1/3$ and that $1/3$ is optimal. The argument also yields a sharp upper bound for $Mf(r)$ in terms of $\lvert f(0) \rvert$, complementing a result due to Bombieri.
\end{abstract}

\date{\today}


\thanks{Research supported by Grant 354537 of the Research Council of Norway.}

\author{Ole Fredrik Brevig} 
\address{Department of Mathematical Sciences, Norwegian University of Science and Technology (NTNU), 7491 Trondheim, Norway} 
\email{ole.brevig@ntnu.no}

\maketitle

This note contains a new proof of a well-known result due to Bohr~\cite{Bohr1914}. Let $f(z)=\sum_{k\geq0} a_k z^k$ be an analytic function in the unit disc $\mathbb{D}$ and consider its majorant function
\[Mf(z) = \sum_{k=0}^\infty \lvert a_k \rvert z^k,\]
so that $\lvert f(z) \rvert \leq Mf(\lvert z\rvert)$ for every $z$ in $\mathbb{D}$. The Schur class $S$ is the set of analytic functions in $\mathbb{D}$ that satisfy $\lvert f(z)\rvert \leq 1$ for every $z$ in $\mathbb{D}$. Inspired by a problem concerning Dirichlet series, Bohr was interested in the quantity
\[B(r) = \sup_{f \in S} Mf(r)\]
for $ 0 \leq r <1$ and, in particular, whether there are $r>0$ such that $B(r)=1$.

The easiest way to get an upper bound for $B$ is to use Cauchy's integral formula to bound $\lvert a_k \rvert \leq 1$, so that
\begin{equation} \label{eq:Brupper}
	B(r) \leq \frac{1}{1-r}.
\end{equation}
The most natural way to get a lower bound for $B(r)$ is to consider the class of disc automorphisms $\operatorname{Aut}(\mathbb{D})$, that is, the functions of the form
\begin{equation} \label{eq:varphi}
	\varphi(z) = \lambda \frac{w-z}{1-\overline{w}z},
\end{equation}
where $\lambda$ is a unimodular constant and $w$ is a point in $\mathbb{D}$. Let $A(r)$ denote the corresponding supremum over $f$ in $\operatorname{Aut}(\mathbb{D})$, so that $B(r) \geq A(r)$. 

\begin{lemma}
	If $0 \leq r \leq 1/3$, then $A(r)=1$. If $1/3 \leq r < 1$, then 
	\[A(r) = \frac{3-\sqrt{8(1-r^2)}}{r}.\]
	In particular, $1 < A(r) < 3r$ for $1/3 < r < 1$.
\end{lemma}

This formula is implicit in the work of Bombieri~\cite{Bombieri1962}, but we repeat some of the details of its proof. We shall only need to know that $A\geq1$ and that $1 < A(r) < 3r$ for $1/3 < r < 1$.

\begin{proof}[Proof of the lemma]
	If $\varphi$ is as in \eqref{eq:varphi}, then 
	\begin{equation} \label{eq:Mvarphi}
		M\varphi(r) = \lvert w \rvert + (1-\lvert w \rvert^2) \sum_{k=1}^\infty \lvert w \rvert^{k-1} r^k = \lvert w \rvert + \frac{(1-\lvert w \rvert^2)r}{1-\lvert w \rvert r}.
	\end{equation}
	The computation of $A(r)$ is now a straightforward calculus exercise. The final two assertions reduce, respectively, to $0<9(r-1/3)^2$ and $0<9r^2-1$, both of which hold for $1/3<r<1$. See Figure~\ref{fig:Ar}.
\end{proof}

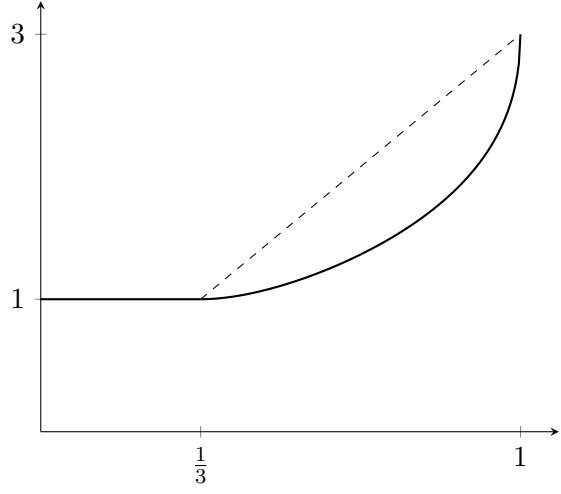
\begin{figure}
	\centering
	\begin{tikzpicture}
	\begin{axis}[axis lines=left, xmin=0, xmax=1.08, ymin=0, ymax=3.25, samples=200,
	             xtick={1/3,1}, xticklabels={$\frac13$,$1$}, ytick={1,3}]
		\addplot[thick, domain=0:1/3] {1};
		\addplot[thick, domain=1/3:1] {(3-sqrt(8*(1-x^2)))/x};
		\addplot[dashed, domain=1/3:1] {3*x};
	\end{axis}
	\end{tikzpicture}
	\caption{$A(r)$ for $0 \leq r < 1$ and $3r$ for $1/3 \leq r < 1$.}
	\label{fig:Ar}
\end{figure}

It is plain that $M(fg)(r) \leq Mf(r)\,Mg(r)$ by the triangle inequality. When combined with the Taylor expansion and the triangle inequality again, we find that $M(\varphi \circ h)(r) \leq M\varphi(Mh(r))$ provided $Mh(r)<1$.

\begin{theorem}
	$B(r) = 1$ if and only if $0 \leq r \leq 1/3$.
\end{theorem}

\begin{proof}
	Let $f$ be a nonconstant function in $S$. Then $f(0)$ is in $\mathbb{D}$ by the maximum principle, so Schur's algorithm and the Schwarz lemma give
	\[f(z) = \frac{f(0)-zg(z)}{1-\overline{f(0)}zg(z)} = \varphi\left(zg(z)\right)\]
	for some $g$ in $S$, where $\varphi$ is as in \eqref{eq:varphi} with $\lambda=1$ and $w=f(0)$. Using the definitions of $B$ and $A$, we infer from this that
	\[Mf(r) \leq M\varphi\left(rMg(r)\right) \leq A(rB(r)),\]
	provided $rB(r)<1$. This bound also trivially holds for constant $f$ in $S$, since $A\geq1$. Taking the supremum, we obtain the bootstrap inequality
	\begin{equation} \label{eq:bootstrap}
		B(r) \leq A(rB(r))
	\end{equation}
	still under the proviso $rB(r)<1$. Let $b=B(1/3)$ and note that $b \leq 3/2$ by \eqref{eq:Brupper}. If $b>1$, then $1/3 < b/3 \leq 1/2$, so \eqref{eq:bootstrap} and the lemma give
	\[b \leq A(b/3) < b,\]
	which is absurd. Hence $b=B(1/3) \leq 1$, and since $B$ is increasing and $B \geq A \geq 1$, we get $B(r)=1$ for $0 \leq r \leq 1/3$.
	
	The bound $B(r) \geq A(r)$ and the lower bound in the final assertion of the lemma complete the proof.
\end{proof}

The bootstrap argument can be run once more. Suppose that $f$ is in $S$ and that $\lvert f(0) \rvert < 1$. Fix $0 \leq r \leq 1/3$ and repeat the first part of the proof with $B(r)=1$ as input, to obtain
\begin{equation} \label{eq:si}
	Mf(r) \leq \lvert f(0) \rvert + \frac{(1-\lvert f(0) \rvert^2)r}{1-\lvert f(0) \rvert r}.
\end{equation}
After recognizing the right-hand side of \eqref{eq:si} as $M\varphi(r)$ from \eqref{eq:Mvarphi}, we may think of \eqref{eq:si} as an $\ell^1$ version of Littlewood's subordination principle, since $f$ is subordinate to $\varphi$. This is in fact a special case of a result due to Bhowmik and Das~\cite{BD2018}*{Lemma~1}, which asserts that $Mf(r) \leq M\varphi(r)$ for $0 \leq r \leq 1/3$ without the requirement that $\varphi$ be a disc automorphism. Their argument relies on Bohr's theorem, while the argument presented above yields Bohr's theorem and \eqref{eq:si} together.

The corresponding $\ell^2$ version is Littlewood's subordination principle. If $0 \leq r < 1$ and if $f$ is subordinate to $\varphi$, then
\begin{equation} \label{eq:LSP}
	\int_0^{2\pi} \lvert f(re^{i\theta}) \rvert^2 \,\frac{d\theta}{2\pi} \leq \int_0^{2\pi} \lvert \varphi(re^{i\theta}) \rvert^2 \,\frac{d\theta}{2\pi}.
\end{equation}
If $\varphi$ is a disc automorphism, Bombieri \cite{Bombieri1962}*{Teorema~B} used \eqref{eq:LSP} to establish that \eqref{eq:si} holds also for $0 \leq r \leq \lvert f(0) \rvert$.\footnote{A different proof of Bombieri's result can be extracted from the proof of \cite{KP2017}*{Theorem~1}.} In combination, we have the following.

\begin{corollary}
	Let $f$ be in $S$ with $\lvert f(0) \rvert <1$. If $0\leq r\leq \max(\lvert f(0) \rvert,1/3)$, then
	\[Mf(r)\leq M\varphi(r),\]
	where $\varphi(z) = \frac{w-z}{1-\overline{w}z}$ for $w=f(0)$.
\end{corollary}

For $f(0)=0$, the range $0\leq r\leq1/3$ in the corollary cannot be improved due to Bohr's theorem. Fix $0<\varrho<1$. It would be interesting to determine the largest $r$ such that $Mf(r)\leq M\varphi(r)$ for every $f$ in $S$ with $\lvert f(0)\rvert=\varrho$.
   
\subsection*{AI disclosure} The approach via Schur's algorithm and the majorant of a composition was identified by the author. The bootstrap argument and the corollary were developed in dialogue with Claude (Opus~5 and Fable~5.1). Claude was also used to check the argument, assist with the presentation, produce the figure, and search the literature (in particular, identifying \cite{BD2018} and \cite{KP2017}). The author wrote the final version of all arguments and takes full responsibility for the content.

\bibliography{bbohr}

\end{document}